\documentclass[10 pt]{amsart}
\usepackage{amssymb}
\usepackage[all]{xy}
\newtheorem{theorem}{Theorem}[]
\newtheorem{lemma}[theorem]{Lemma}
\newtheorem{definition}[theorem]{Definition}
\newtheorem{question}[theorem]{Question}
\newtheorem{remark}[theorem]{Remark}
\def\H{\mathfrak{M}}
\def\o{\mathfrak{o}}
\def\m{\mathfrak{r}}
\def\Pa{\mathfrak{P}}
\def\C{\mathbf{C}}
\def\P{\Phi}
\def\S{\Psi}
\def\L{\Lambda}
\def\G{\Gamma}
\def\om{\omega}
\def\to{\longrightarrow}
\def\ot{\otimes}
\def\h{\hspace}
\begin{document}
\baselineskip14pt
\title[Quantum Functor $\mathfrak{M}\mathfrak{o}\mathfrak{r}$]{Quantum Functor $\mathfrak{M}\mathfrak{o}\mathfrak{r}$}
\author[M. M. Sadr]{Maysam Maysami Sadr}
\email{sadr@iasbs.ac.ir}
\address{Department of  Mathematics, Institute for Advanced Studies in Basic Sciences (IASBS),
P. O. Box 45195-1159 Zanjan, Iran.}
\subjclass[2000]{46L05, 46L85, 46M20}
\keywords{Category of C*-algebras, quantum family of maps, non commutative algebraic topology}
\begin{abstract}
Let ${\bf Top}_c$ be the category of compact spaces and continuous maps and
${\bf Top}_f\subset{\bf Top}_c$ be the full subcategory of finite spaces. Consider the covariant functor
$Mor:{\bf Top}_f^{op}\times{\bf Top}_c\to{\bf Top}_c$ that associates any pair $(X,Y)$ with the space of all morphisms
from $X$ to $Y$. In this paper, we describe a non commutative version of $Mor$. More pricelessly, we define
a functor $\H\o\m$, that takes any pair $(B,C)$ of a finitely generated unital C*-algebra $B$ and a finite dimensional
C*-algebra $C$ to the quantum family of all morphism from $B$ to $C$. As an application we introduce a non commutative version of path functor.
\end{abstract}
\maketitle
\section{Introduction}
Let $\bf Top$ be the category of topological spaces and continuous maps. For every $X,Y\in{\bf Top}$, denote
by $Mor(X,Y)$ the set of all morphisms (continuous maps) from $X$ to $Y$. Also for spaces $X_1,X_2,Y_1,Y_2$ and
morphisms $f:X_2\to X_1$, $g:Y_1\to Y_2$, denote by $Mor(f,g)$ the map $h\longmapsto ghf$ from $Mor(X_1,Y_1)$ to
$Mor(X_2,Y_2)$. Then $Mor$ can be considered as a covariant functor from the product category
${\bf Top}^{op}\times{\bf Top}$ to $\bf Top$, where $Mor(X,Y)$ has the compact-open topology. Now, suppose that
${\bf Top}_c\subset{\bf Top}$ and ${\bf Top}_f\subset{\bf Top}$ are the full subcategories of compact Hausdorff
spaces and finite discrete spaces, respectively. Then the restriction of $Mor$ to ${\bf Top}_f^{op}\times{\bf Top}_c$
takes its values in ${\bf Top}_c$:
$$Mor:{\bf Top}_f^{op}\times{\bf Top}_c\to{\bf Top}_c.$$
The aim of this paper is description of a {\it quantum} ({\it non commutative}) version of $Mor$. We define a
covariant functor
$$\H\o\m:\C^*_{fg}\times{\C^*_{fd}}^{op}\to\C^*_{fg},$$
where $\C^*_{fg}$ is the category of finitely generated unital C*-algebras and $\C^*_{fd}$ is the full subcategory
of finite dimensional C*-algebras, such that for C*-algebras $B,C$, $\H\o\m(B,C)$ is the quantum family of all
morphisms from $B$ to $C$. For clearness of the idea behind this definition, we recall some basic
terminology of Non Commutative Geometry.

Let $\C^*$ be the category of unital C*-algebras and unital *-homomorphisms
and let $\C^*_{com}$ be the full subcategory of commutative algebras. For every $X\in{\bf Top}_c$ and
$A\in\C^*_{com}$, let $\mathfrak{f}X$ and $\mathfrak{q}A$ be the C*-algebra of continuous complex valued maps on
$X$ and the spectrum of $A$ with w* topology, respectively. Then the famous Gelfand Theorem says that
${\bf Top}_c$ and $\C^*_{com}$ are dual of each other, under functors $\mathfrak{f}$ and $\mathfrak{q}$. Thus for
every $A\in\C^*$, one can consider a {\it symbolic} notion $\mathfrak{q}A$ of a quantum or non commutative space
that its algebra of functions ($\mathfrak{f}\mathfrak{q}A$) is $A$. In this correspondence, for $A\in\C^*_{fg}$
and $B\in\C^*_{fd}$, $\mathfrak{q}A$ and $\mathfrak{q}B$ are called {\it finite dimensional} compact quantum
space and {\it discrete finite} quantum space, respectively.

Now, suppose that $X,Y,Z$ are topological spaces and $f:X\to Mor(Z,Y)$ is a continuous map. One can redefine
the map $f$ in a useful manner:
$$f:X\times Z\to Y\quad\quad(x,z)\longmapsto f(x)(z).$$
Thus, one can consider a {\it family} $\{f(x)\}_{x\in X}$ of maps from $Z$ to $Y$ {\it parameterized} by
$f$ and with {\it parameters} $x$ in $X$, as a map from $X\times Z$ to $Y$. Using this, Woronowicz \cite{W1}
defined the notion of {\it quantum family of maps} from quantum space $\mathfrak{q}C$ to quantum space
$\mathfrak{q}B$ as a pair $(\mathfrak{q}A,\P)$, where $\P:B\to C\ot A$ is a *-homomorphism and $\ot$ denotes the
minimal tensor product of C*-algebras. (Note that Woronowicz has used the terminology "pseudo space" instead of
"quantum space". Our terminology here is from Soltan's paper \cite{S1}.) Also, using a universal property,
he defined the notion of quantum family of {\it all} maps, that may or may not exists in general case, see
Section 2.

Now for the definition of $\H\o\m$, we let $\H\o\m(B,C)$ be the quantum family of all maps from $\mathfrak{q}C$
to $\mathfrak{q}B$.

In the following sections, we use the terminology "quantum family of morphisms from $B$ to $C$"
instead of "quantum family of maps from $\mathfrak{q}C$ to $\mathfrak{q}B$".
\section{Quantum family of morphisms}
Let $B,B',C,C'$ be unital C*-algebras. We denote by $B\ot C$ and $B\oplus C$  the minimal tensor product and
direct sum, respectively. For *-homomorphisms $\P:B\to C$ and $\P':B'\to C'$, $\P\ot\P'$ denotes the natural
homomorphism from $B\ot B'$ to $C\ot C'$ defined by $b\ot b'\to\P(b)\ot\P'(b')$ ($b\in B,b'\in B'$).
Also, let $\P\oplus \P':B\oplus B'\to C\oplus C'$ be the homomorphism defined by $\P\oplus \P'(b,b')=(\P(b),\P'(b'))$.
Denote by $C^\circ$, the space of all bounded functionals on $C$, and by $1_C$ the unite element of $C$.

A quantum family $(A,\P)$ of morphisms from $B$ to $C$ consists of a
unital C*-algebra $A$ and a unital *-homomorphism $\P:B\to C\ot A$.
\begin{definition}
Let $B,C$ and $(A,\P)$ be as above. Then $(A,\P)$ is called a quantum family of all morphisms from $B$ to $C$
if for every unital C*-algebra $D$ and any unital *-homomorphism $\S:B\to C\ot D$, there is a unique unital
*-homomorphism $\L:A\to D$ such that the following diagram is commutative:

\[\xymatrix{B\ar[rr]^-{\P}\ar@{=}[d]&& C\ot A\ar[d]^{id_C\ot\L}\\
B\ar[rr]^-{\S}&& C\ot D}\]

\end{definition}
If $(A,\P)$ and $(A',\P')$ are two quantum families of all morphisms from $B$ to $C$, then by the above
universal property there is a isometric *-isomorphism between $A$ and $A'$.
\begin{theorem}\label{t1}
Let $B$ be a finitely generated unital C*-algebra and $C$ be a finite dimensional C*-algebra. Then the
quantum family $(A,\P)$ of all morphisms from $B$ to $C$ exists. Also $A$ is finitely generated
(and unital) and the set $G=\{(\om\ot id_A)\P(b): b\in B, \om\in C^\circ\}$ is a generator for $A$ (the smallest
closed *-subalgebra of $A$ containing $G$ is equal to $A$).
\end{theorem}
\begin{proof}
See \cite{S1}.
\end{proof}
For examples of quantum families of morphisms see, \cite{W1} and \cite{S1}.
\section{Definition of the functor}
As previously, let $\C^*$ be the category of unital C*-algebras and unital homomorphisms. Also, let
$\C_{fg}^*\subset\C^*$ and $\C^*_{fd}\subset\C^*$ be the full subcategories of finitely generated and
finite dimensional C*-algebras, respectively. For $B_1,B_2\in\C^*$, denote by $Mor(B_1,B_2)$ the set
of all morphisms from $B_1$ to $B_2$ in $\C^*$. For more details on the category of C*-algebras, see \cite{W1}.
Note that by elementary results of the theory of C*-algebras, a morphism $f\in Mor(B_1,B_2)$ is an
isomorphism in the categorical sense (i.e. there is $g\in Mor(B_2,B_1)$ such that $gf=id_{B_1}$ and
$fg=id_{B_2}$) if and only if $f$ is an isometric *-isomorphism from $B_1$ onto $B_2$.

For any $B\in\C_{fg}^*$ and every $C\in\C^*_{fd}$, let $(\H\o\m(B,C),\Pa_{B,C})$ be the quantum family of
all morphisms from $B$ to $C$ ($\Pa$ stands for "parameter").

For every  $B_1,B_2\in\C_{fg}^*$, $C_1,C_2\in\C^*_{fd}$ and $f\in Mor(B_1,B_2)$, $g\in Mor(C_2,C_1)$,
let $\H\o\m(f,g)$ be the unique morphism from $\H\o\m(B_1,C_1)$ to $\H\o\m(B_2,C_2)$ in $\C_{fg}^*$
such that the following diagram is commutative:
\begin{equation}\label{d4}
\xymatrix{B_1\ar[rrrr]^-{\Pa_{B_1,C_1}}\ar[d]^{f}&& &&C_1\ot \H\o\m(B_1,C_1)\ar[d]^{id_{C_1}\ot \H\o\m(f,g)}\\
B_2\ar[rrrr]^-{(g\ot id_{\H\o\m(B_2,C_2)})\Pa_{B_2,C_2}}\ar@{=}[d]&& &&C_1\ot \H\o\m(B_2,C_2)\\
B_2\ar[rrrr]^-{\Pa_{B_2,C_2}}&& &&C_2\ot\H\o\m(B_2,C_2)\ar[u]^{g\ot id_{\H\o\m(B_2,C_2)}}}
\end{equation}
Let $B_3\in\C_{fg}^*$, $C_3\in\C^*_{fd}$  and $f'\in Mor(B_2,B_3)$, $g'\in Mor(C_3,C_2)$,
then by the definition we have the following commutative diagram:
\begin{equation}\label{d5}
\xymatrix{B_2\ar[rrrr]^-{\Pa_{B_2,C_2}}\ar[d]^{f'}&& && C_2\ot \H\o\m(B_2,C_2)\ar[d]^{id_{C_2}\ot \H\o\m(f',g')}\\
B_3\ar[rrrr]^-{(g'\ot id_{\H\o\m(B_3,C_3)})\Pa_{B_3,C_3}}&& && C_2\ot \H\o\m(B_3,C_3).}
\end{equation}
Then we have,
\begin{equation*}
\begin{split}
&(id_{C_1}\ot\H\o\m(f',g')\H\o\m(f,g))\Pa_{B_1,C_1}\\
=&(id_{C_1}\ot\H\o\m(f',g'))(id_{C_1}\ot\H\o\m(f,g))\Pa_{B_1,C_1}\\
=&(id_{C_1}\ot\H\o\m(f',g'))(g\ot id_{\H\o\m(B_2,C_2)})\Pa_{B_2,C_2}f\h{8mm}\text{(by \ref{d4})}\\
=&(g\ot id_{\H\o\m(B_3,C_3)})(id_{C_2}\ot\H\o\m(f',g'))\Pa_{B_2,C_2}f\\
=&(g\ot id_{\H\o\m(B_3,C_3)})(g'\ot id_{\H\o\m(B_3,C_3)}\Pa_{B_3,C_3}f')f\h{9mm}\text{(by \ref{d5})}\\
=&((gg')\ot id_{\H\o\m(B_3,C_3)})\Pa_{B_3,C_3}(f'f).
\end{split}
\end{equation*}
Thus by the uniqueness property of the definition of $\H\o\m(f'f,gg')$ we have,
$$\H\o\m(f'f,gg')=\H\o\m(f',g')\H\o\m(f,g).$$
Also, it is clear that
$$\H\o\m(id_B,id_C)=id_{\H\o\m(B,C)},$$
for every $B\in\C_{fg}^*$ and $C\in\C^*_{fd}$. Thus we have defined a covariant functor
$$\H\o\m:\C_{fg}^*\times\C^{*^{op}}_{fd}\to\C_{fg}^*$$
that is called Quantum Functor $\H\o\m$. We often write $\H$ instead of $\H\o\m$.
\section{Some properties}
In this section we prove some basic properties of the functor $\H\o\m$.
We need the following simple lemma.
\begin{lemma}\label{l1}
Let $A,A',B,B',C$ be C*-algebras, $\P,\P',\L,\G$ be *-homomorphisms and $\om\in C^\circ$ such that
the following diagram is commutative:
\[\xymatrix{B\ar[rr]^-{\P}\ar[d]^{\L}&& C\ot A \ar[d]^{id_C\ot\G}\\
B'\ar[rr]^-{\P'}&& C\ot A' }\]
Then for any $b\in B$, we have
$$\G(\om\ot id_A)\P(b)=(\om\ot id_{A'})\P'\L(b).$$
\end{lemma}
\begin{proof}
By commutativity of the diagram we have
$$(id_C\ot\G)\P(b)=\P'\L(b),$$
and thus
$$(\om\ot id_{A'})(id_C\ot\G)\P(b)=(\om\ot id_{A'})\P'\L(b).$$
Then the left hand side of the latter equation is equal to $\G(\om\ot id_A)\P(b)$, since
$$(\om\ot id_{A'})(id_C\ot\G)=\G(\om\ot id_A).$$
\end{proof}
\begin{theorem}
Let $B\in\C_{fg}^*$ and $C_1,C_2\in\C^*_{fd}$. Then $\H\o\m(B,C_1\ot C_2)$ and $\H\o\m(\H\o\m(B,C_1),C_2)$
are canonically isometric *-isomorphic.
\end{theorem}
\begin{proof}
Let
$$\S:\H(B,C_1\ot C_2)\to C_1\ot C_2\ot\H(\H(B,C_1),C_2)$$
be the unique morphism such that the following diagram is commutative:
\begin{equation}\label{d1}
\xymatrix{B\ar[rrrr]^-{\Pa_{B,C_1\ot C_2}}\ar[d]^{\Pa_{B,C_1}}&& &&
C_1\ot C_2\ot\H(B,C_1\ot C_2)\ar[d]^{id_{C_1\ot C_2}\ot\S}\\
C_1\ot\H(B,C_1)\ar[rrrr]^-{id_{C_1}\ot\Pa_{\H(B,C_1),C_2}}&& && C_1\ot C_2\ot\H(\H(B,C_1),C_2).}
\end{equation}
Suppose that the morphisms
$$\G:\H(B,C_1)\to C_2\ot\H(B,C_1\ot C_2)$$
and
$$\S':\H(\H(B,C_1),C_2)\to\H(B,C_1\ot C_2)$$
are the unique morphisms such that the following diagrams are commutative.
\begin{equation}\label{d2}
\xymatrix{B\ar[rrrr]^-{\Pa_{B,C_1}}\ar@{=}[d]&& && C_1\ot\H(B,C_1)\ar[d]^{id_{C_1}\ot\G}\\
B\ar[rrrr]^-{\Pa_{B,C_1\ot C_2}}&& && C_1\ot C_2\ot\H(B,C_1\ot C_2),}
\end{equation}
\begin{equation}\label{d3}
\xymatrix{\H(B,C_1)\ar[rrrr]^-{\Pa_{\H(B,C_1),C_2}}\ar@{=}[d]&& &&
C_2\ot\H(\H(B,C_1),C_2)\ar[d]^{id_{C_2}\ot\S'}\\
\H(B,C_1)\ar[rrrr]^-{\G}&& && C_2\ot\H(B,C_1\ot C_2).}
\end{equation}
Then commutativity of Diagram \ref{d3}, implies that
\begin{equation}\label{e1}
id_{C_1}\ot\G=(id_{C_1\ot C_2}\ot\S')(id_{C_1}\ot\Pa_{\H(B,C_1),C_2}).
\end{equation}
Now, we have
\begin{equation*}
\begin{split}
&(id_{C_1\ot C_2}\ot(\S'\S))\Pa_{B,C_1\ot C_2}\\
=&(id_{C_1\ot C_2}\ot\S')(id_{C_1\ot C_2}\ot\S)\Pa_{B,C_1\ot C_2}\\
=&(id_{C_1\ot C_2}\ot\S')(id_{C_1}\ot\Pa_{\H(B,C_1),C_2})\Pa_{B,C_1}\h{10mm}\text{(by \ref{d1})}\\
=&(id_{C_1}\ot\G)\Pa_{B,C_1}\h{50mm}\text{(by \ref{e1})}\\
=&\Pa_{B,C_1\ot C_2}.\h{59.5mm}\text{(by \ref{d2})}
\end{split}
\end{equation*}
Thus by the universal property of quantum family of all morphisms, we have
\begin{equation}\label{e2}
\S'\S=id_{\H\o\m(B,C_1\ot C_2)}
\end{equation}
Now, we show that
\begin{equation}\label{e3}
(id_{C_2}\ot\S)\G=\Pa_{\H\o\m(B,C_1),C_2}.
\end{equation}
By Theorem \ref{t1}, it is sufficient to prove
$$(id_{C_2}\ot\S)\G(a)=\Pa_{\H(B,C_1),C_2}(a),$$
where $a=(\om\ot id_{\H(B,C_1)})\Pa_{B,C_1}(b)$ for $\om\in C^\circ_1$ and $b\in B$. We have
\begin{equation*}
\begin{split}
&(id_{C_2}\ot\S)\G(a)\\
=&(id_{C_2}\ot\S)\G(\om\ot id_{\H(B,C_1)})\Pa_{B,C_1}(b)\\
=&(id_{C_2}\ot\S)(\om\ot\Pa_{B,C_1\ot C_2})(b)\h{45.5mm}\text{(by \ref{d2})}\\
=&(\om\ot id_{C_2\ot\H(\H(B,C_1),C_2)})(id_{C_1}\ot\Pa_{\H(B,C_1),C_2})
\Pa_{B,C_1}(b)\h{5mm}\text{(by Lemma \ref{l1})}\\
=&\Pa_{\H(B,C_1),C_2}(\om\ot id_{\H(B,C_1)})\Pa_{B,C_1}(b)\\
=&\Pa_{\H(B,C_1),C_2}(a).
\end{split}
\end{equation*}
Now, we have
\begin{equation*}
\begin{split}
&(id_{C_2}\ot(\S\S'))\Pa_{\H(B,C_1),C_2}\\
=&(id_{C_2}\ot\S)(id_{C_2}\ot\S')\Pa_{\H(B,C_1),C_2}\\
=&(id_{C_2}\ot\S)\G\h{60mm}\text{(by \ref{d3})}\\
=&\Pa_{\H(B,C_1),C_2}\h{60mm}\text{(by \ref{e3})}.
\end{split}
\end{equation*}
Thus by the universal property of $\H\o\m$, we have
\begin{equation}\label{e4}
\S\S'=id_{\H(\H(B,C_1),C_2)}.
\end{equation}
At last, \ref{e2} and \ref{e4} show that  $\H\o\m(B,C_1\ot C_2)$ and $\H\o\m(\H\o\m(B,C_1),C_2)$
are isomorphic in $\C^*$.
\end{proof}
The above Theorem corresponds to the following fact in ${\bf Top}_{c}$:
\\{\it Let $X_1,X_2$ and $Y$ be  compact Hausdorff spaces. Then the map
$$F:Mor(X_1\times X_2,Y)\to Mor(X_1,Mor(X_2,Y)),$$
defined by $(F(f))(x_1)(x_2)=f(x_1,x_2)$ for $f\in Mor(X_1\times X_2,Y),x_1\in X_1,x_2\in X_2$, is a homeomorphism
of topological spaces.}
\\The proof of this topological fact is elementary. For some general results on this type, see \cite{J}.

Let $X,Y_1,Y_2$ be in ${\bf Top}$ and let $f:Y_1\to Y_2$ be an injective continuous map. Then the morphism
$Mor(id_X,f):Mor(X,Y_1)\to Mor(X,Y_2)$, defined by $g\longmapsto fg$ is also injective. Analogously, in $\C^*$
we have:
\begin{theorem}
Let $B_1,B_2\in\C_{fg}^*$, $C\in\C^*_{fd}$ and let $f$ be in $Mor(B_1,B_2)$. Suppose that $f$ is a surjective map.
Then $\H\o\m(f,id_C)$ is also surjective.
\end{theorem}
\begin{proof}
By Theorem \ref{t1}, the set
$$G=\{(\om\ot id_{\H(B_2,C)})\Pa_{B_2,C}(b):\quad b\in B_2,\om\in C^\circ\}$$
is a generator for $\H(B_2,C)$. For every $b_2\in B_2$ there is a $b_1\in B_1$ such that $f(b_1)=b_2$.
Thus by Lemma \ref{l1}, we have
$$G=\{\H(f,id_C)(\om\ot id_{\H(B_1,C)})\Pa_{B_1,C}(b):\quad b\in B_1,\om\in C^\circ\}.$$
Thus $\H\o\m(f,id_C):\H\o\m(B_1,C)\to\H\o\m(B_2,C)$ is surjective, since the image of every *-homomorphism
between C*-algebras is closed.
\end{proof}
\begin{theorem}
Let $B_1,B_2\in\C_{fg}^*$ and $C_1,C_2\in\C^*_{fd}$, then there is a canonical morphism $\S$ from
$\H\o\m(B_1\oplus B_2,C_1\oplus C_2)$ onto $\H\o\m(B_1,C_1)\oplus\H\o\m(B_2,C_2)$.
\end{theorem}
\begin{proof}
Let $B=B_1\oplus B_2$ and $C=C_1\oplus C_2$. Let
$$\G:(C_1\ot\H(B_1,C_1))\oplus(C_2\ot\H(B_2,C_2))\to C\ot(\H(B_1,C_1)\oplus\H(B_2,C_2))$$
be the natural embedding. Then, suppose that $\S$ is the unique morphism such that
the following diagram is commutative:
\[\xymatrix{B\ar[rr]^-{\Pa_{B,C}}\ar[d]^{\Pa_{B_1,C_1}\oplus\Pa_{B_2,C_2}}&&  C\ot\H(B,C)\ar[d]^{id_{C}\ot\S}\\
(C_1\ot\H(B_1,C_1))\oplus(C_2\ot\H(B_2,C_2))\ar[rr]^-{\G}&&  C\ot(\H(B_1,C_1)\oplus\H(B_2,C_2)).}\]
\\Now we prove that $\S$ is surjective. By Theorem \ref{t1}, the C*-algebra $\H(B_1,C_1)\oplus\H(B_2,C_2)$ is
generated by the set $G$ of all pairs $(a_1,a_2)$ where
$$a_1=(\om_1\ot id_{\H(B_1,C_1})\Pa_{B_1,C_1}(b_1)\quad\text{and}\quad
a_2=(\om_2\ot id_{\H(B_2,C_2})\Pa_{B_2,C_2}(b_2)$$
for $b_1\in B_1, b_2\in B_2, \om_1\in C^\circ_1,\om_2\in C^\circ_2$. It is easily checked that such pairs
are in the form of
$$(a_1,a_2)=((\om_1\oplus\om_2)\ot id_{\H(B_1,C_1)\oplus\H(B_2,C_2)})\G(\Pa_{B_1,C_1}\oplus\Pa_{B_2,C_2})(b_1,b_2),$$
and thus by Lemma \ref{l1}, we can write
$$(a_1,a_2)=\S((\om_1\oplus\om_2)\ot id_{\H(B,C)})\Pa_{B,C}(b_1,b_2).$$
Thus the generator set $G$ is in the image of $\S$ and therefore $\S$ is surjective.
\end{proof}
Similarly one can prove the following.
\begin{theorem}
Let $B_1,\cdots,B_n\in\C_{fg}^*$ and let $C_1,\cdots,C_n\in\C^*_{fd}$. Then there is a
canonical surjective *-homomorphism
$$\S:\H\o\m(B_1\oplus\cdots\oplus B_n,C_1\oplus\cdots\oplus C_n)\to\H\o\m(B_1,C_1)\oplus\cdots\oplus\H\o\m(B_n,C_n).$$
\end{theorem}
The construction appearing in the following Theorem, is a special case of the notion of {\it composition}
of quantum families of maps, defined in \cite{S1}.
\begin{theorem}
Let $B\in\C_{fg}^*$ and $C,D\in\C^*_{fd}$. Then there is a canonical morphism
$$\S:\H\o\m(B,C)\to\H\o\m(C,D)\ot\H\o\m(B,C).$$
\end{theorem}
\begin{proof}
The desired morphism $\S$ is the unique morphism such that the following diagram becomes commutative:
\[\xymatrix{B\ar[rrrr]^-{\Pa_{B,D}}\ar[d]^{\Pa_{B,C}}&& && D\ot\H(B,D)\ar[d]^{id_{D}\ot\S}\\
C\ot\H(B,C)\ar[rrrr]^-{\Pa_{C,D}\ot id_{\H(B,C)}}&& && D\ot\H(C,D)\ot\H(B,C).}\]
\end{proof}
By the latter assumptions, let $B=C=D=M$ be in $\C^*_{fd}$. Then it is proved in \cite{S1}, that the map
$$\S:\H\o\m(M,M)\to\H\o\m(M,M)\ot\H\o\m(M,M)$$
is a coassociative comultiplication (i.e. $(id_{\H(M,M)}\ot\S)\S=(\S\ot id_{\H(M,M)})\S$) and thus the pair
$(\H(M,M),\S)$ is a compact quantum semigroup.
\begin{theorem}\label{t2}
Let $(\{B_i\},\{\P_{ij}\})_{i\leq j\in I}$ be a diredcted system in $\C^*_{fg}$, on the directed set $I$, such that
$B=\varinjlim B_i\in\C^*_{fg}$. Suppose that $C\in\C^*_{fd}$, then the canonical morphism from
$A=\varinjlim\H\o\m(B_i,C)$ to $\H\o\m(B,C)$ is surjective.
\end{theorem}
\begin{proof}
For every $i\in I$, let $f_i:B_i\to B$ and $g_i:\H(B_i,C)\to A$ be the canonical morphisms of direct limit
structures. Consider the directed system
$$(\{\H(B_i,C)\},\{\H(\P_{ij},id_C)\})_{i\leq j\in I}.$$
Then, for every $i\leq j$, we have
$$\H(f_i,id_C)=\H(f_j,id_C)\H(\P_{ij},id_C).$$
Thus by the universal property of direct limit, there is a unique morphism $\S:A\to\H(B,C)$ such that
\begin{equation}\label{e7}
\S g_i=\H(f_i,id_C)
\end{equation}
for every $i\in I$. We must prove that $\S$ is surjective. By Theorem \ref{t1}, the set
$$G=\{(\om\ot id_{\H(B,C)})\Pa_{B,C}(b): b\in B,\om\in C^\circ\}$$
is a generator for $\H(B,C)$. Let $b\in B$ and $\om\in C^\circ$ be arbitrary and fixed. There are $i\in I$ and $b_i\in B_i$ such that $f_i(b_i)=b$. Let
$$a_i=(\om\ot id_{\H(B_i,C)})\Pa_{B_i,C}(b_i)\in\H(B_i,C),$$
then by Lemma \ref{l1}, $\H(f_i,id_C)(a_i)=b$. Now, by (\ref{e7}), we have $\S g_i(a_i)=b$.
Thus $G$ is in the image of $\S$, and $\S$ is surjective.
\end{proof}
\begin{question}
Is the map $\S$, constructed in Theorem \ref{t2}, injective?
\end{question}
\begin{theorem}\label{t3}
Let $C$ be a commutative finite dimensional C*-algebra and let $B_1,B_2\in\C^*_{fg}$. Then there is a canonical
surjective morphism
$$\S:\H\o\m(B_1\ot B_2,C)\to\H\o\m(B_1,C)\ot\H\o\m(B_2,C).$$
\end{theorem}
\begin{proof}
Let $m:C\ot C\to C$ be defined by $m(c_1,c_2)=c_1c_2$ for $c_1,c_2\in C$. Since $C$ is commutative, $m$ is a morphism.
Also, let
$$F:C\ot\H(B_1,C)\ot C\ot\H(B_2,C)\to C\ot C\ot\H(B_1,C)\ot\H(B_2,C)$$
be the flip map (i.e. $c_1\ot x_1\ot c_2\ot x_2\longmapsto c_1\ot c_2\ot x_1\ot x_2$) and let $\S$ be the unique
morphism such that the following diagram is commutative:
\[\xymatrix{B_1\ot B_2\ar[rr]^-{\Pa_{B_1\ot B_2,C}}\ar[d]^{\Pa_{B_1,C}\ot\Pa_{B_2,C}}&&
C\ot\H(B_1\ot B_2,C)\ar[d]^{id_{C}\ot\S}\\
C\ot\H(B_1,C)\ot C\ot\H(B_2,C)\ar[rr]^-{mF}&&  C\ot\H(B_1,C)\ot\H(B_2,C).}\]
\\Now, we show that $\S$ is surjective. By Theorem \ref{t1}, the set $G_1$ and $G_2$, defined by
$$G_1=\{(\om\ot id_{\H(B_1,C)})\Pa_{B_1,C}(b_1): b_1\in B_1,\om\in C^\circ\}$$
and
$$G_2=\{(\om\ot id_{\H(B_2,C)})\Pa_{B_2,C}(b_2): b_2\in B_2,\om\in C^\circ\}$$
are generator sets of $\H(B_1,C)$ and $\H(B_2,C)$, respectively. Thus the set
$$G=\{a_1\ot1_{\H(B_2,C)},1_{\H(B_1,C)}\ot a_2: a_1\in G_1,a_2\in G_2\}$$
is a generator for $\H(B_1,C)\ot\H(B_2,C)$. On the other hand, for $b_1\in B_1$ and $\om\in C^\circ$, we have
\begin{equation*}
\begin{split}
&((\om\ot id_{\H(B_1,C)})\Pa_{B_1,C}(b_1))\ot1_{\H(B_2,C)}\\
=&(\om\ot id_{\H(B_1,C)\ot\H(B_2,C)})mF(\Pa_{B_1,C}\ot\Pa_{B_2,C})(b_1\ot1_{\H(B_2,C)})\\
=&\S(\om\ot id_{\H(B_1\ot B_2,C)})\Pa_{B_1\ot B_2,C}(b_1\ot1_{\H(B_2,C)})\h{25mm}\text{(by Lemma \ref{l1})}.
\end{split}
\end{equation*}
\\Thus for every $a_1\in G_1$, $a_1\ot1_{\H(B_2,C)}$ is in the image of $\S$. Similarly, for every $a_2\in G_2$,
$1_{\H(B_1,C)}\ot a_2$ is also in the image of $\S$. Therefore $G$ is in the image of $\S$, and $\S$ is surjective.
\end{proof}
The above result, is similar to the following trivial fact in $\bf Top$:
\\{\it Let $Y_1,Y_2,X$ be topological spaces, then the spaces $Mor(X,Y_1\times Y_2)$ and $Mor(X,Y_1)\times Mor(X,Y_2)$
are homeomorphic.}
\begin{question}
Is the map $\S$ of Theorem \ref{t3}, injective?
\end{question}
\section{Symbolic quantum family of paths}
In this section we make a suggestion  for more research.
Let $B\in\C_{fg}^*$ and $C\in\C^*_{fd}$. We have a contravariant functor
$$\H\o\m(B,-):\C^*_{fd}\to\C_{fg}^*,$$
and a covariant functor
$$\H\o\m(-,C):\C_{fg}^*\to\C_{fg}^*.$$
Thus one can consider any contravariant functor $\mathfrak{F}:\C^*_{fd}\to\C_{fg}^*,$ (resp. covariant functor
$\mathfrak{G}:\C_{fg}^*\to\C_{fg}^*$) as a generalized notion of a unital (resp. finite dimensional) C*-algebra.
Note that since $\H\o\m(B,\mathbb{C})\cong B$, where $\mathbb{C}$ denotes the C*-algebra of complex numbers,
one can recover the C*-algebra $B$ from the data of the functor $\H\o\m(B,-)$.

Recall from elementary Algebraic Topology (\cite{Sp}), that for every $X\in{\bf Top}$
the {\it cylinder} space $\mathfrak{c}X$ of $X$ is the space $X\times I$
with product topology and the {\it path} space $\mathfrak{p}X$ of $X$ is the space $Mor(I,X)$
of all continuous maps from $I$ to $X$ with compact open topology, where $I$ is the interval $0\leq r\leq1$.
The covariant functors $\mathfrak{c}$ and $\mathfrak{p}$
are {\it adjoint} functors on $\bf Top$:
\begin{equation}\label{e5}
Mor(\mathfrak{c}X,Y)\cong Mor(X,\mathfrak{p}Y),
\end{equation}
for every topological spaces $X$ and $Y$. For every unital C*-algebra $A$ the cylinder $\mathfrak{c}A$
of $A$ is the C*-algebra of all continuous maps $f:I\to A$, or equivalently $\mathfrak{f}I\ot A$.
The cylinder functors $\mathfrak{c}:{\bf Top}_c\to{\bf Top}_c$
and $\mathfrak{c}:\C^*\to\C^*$ are {\it compatible} with respect to Gelfand's duality,
that is for every $X\in{\bf Top}_c$ and $A\in\C^*_{com}$ we have:
$$\mathfrak{c}\mathfrak{f}X=\mathfrak{f}\mathfrak{c}X,\h{20mm}\mathfrak{c}\mathfrak{q}A=\mathfrak{q}\mathfrak{c}A.$$

As is indicated  by Alain Connes there is no interesting notion of {\it path} or {\it loop} in
non commutative (quantum) spaces (page 544 of \cite{C}). In fact, the natural notion for paths in a quantum
space $\mathfrak{q}A$ is  nonzero *-homomorphisms from $A$ to $\mathfrak{f}I$. But such *-homomorphisms  do not exist
for almost all noncommutative spaces.
Using the functors  $\H\o\m$ and $\mathfrak{c}$, the generalized notion of C*-algebras as functors, and
Identity (\ref{e5}), we {\it symbolically} define the concept of quantum space of
{\it paths} in a finite quantum space $\mathfrak{q} C$:
If we replace $Mor$, $X$, $Y$ and $\mathfrak{c}X$ with the functor
$\H\o\m$, unital C*-algebra $B$, finite dimensional C*-algebra $C$, and the cylinder $\mathfrak{c}B$ of $B$, respectively, then we have {\it formally},
\begin{equation}\label{e6}
\H\o\m(\mathfrak{c}B,C)\cong\H\o\m(B,\mathfrak{p}C).
\end{equation}
Therefore one can consider the covariant functor
$$\H\o\m(\mathfrak{c}-,C):\C^*\to\C^*$$
as a {\it symbolic} definition of $\mathfrak{p}C$.
\begin{remark}
\item[i)] In the above construction, we lost the natural duality between ${\bf Top}_c$ and $\C^*_{com}$
(note with a more care the replacement of the notations in Formulas  (\ref{e5}) and (\ref{e6})).
In the literature of Non Commutative Algebraic Topology, this type of constructions are called
{\it wrong-way} functorial.
\item[ii)] Using suspension of C*-algebras instead cylinder functor in (\ref{e6}), one
can derive a formal definition for the non commutative loop space. But there is a gap in the transformation
from commutative to non commutative case. The problem is that, in the commutative case,
the (reduced) suspension functor and the loop functor are adjoint on the category of
pointed topological spaces and point preserving maps instead of ${\bf Top}$.
Also the standard notion of suspension for C*-algebras involves non unital C*-algebras.
\item[iii)]There are some other notions of loop and path space for algebras \cite{K},
that are very far from our construction.
\end{remark}
\bibliographystyle{amsplain}

\end{document}